\documentclass{amsart}
\usepackage{bbm} % For ``vector of ones`` font, \matbbm{1}
\usepackage{diagbox,amsthm,enumerate}
\usepackage{tikz}
\usepackage{geometry}
\usepackage{microtype,mathtools}
\usepackage[capitalize]{cleveref}
\crefname{equation}{}{}
\Crefname{equation}{}{}

\makeatletter
   \def\marginparright{\@mparswitchfalse}
\makeatother
\usepackage{todonotes}\marginparright\usepackage{csquotes}
\usepackage{lscape}
\usepackage{longtable}\usepackage{arydshln}
\newtheorem{example}{Example}
\newtheorem{remark}{Remark}
\newtheorem{definition}{Definition}
\newtheorem{lemma}{Lemma}
\newtheorem{theorem}{Theorem}
\newcounter{romannumber}
\newenvironment{romannum}
               {\begin{list}{{(\roman{romannumber})}}{\usecounter{romannumber}
                \setlength{\leftmargin}{0pt}
                \setlength{\itemindent}{42pt}}}{\end{list}}
\newlength{\ml}
\usetikzlibrary{shapes,shadings,patterns}
\newcommand{\tu}{
  \tikz [
    grow=up,
    level distance=1.0\ml,
    sibling distance=1.0\ml,	
    every node/.style={shading=ball, ball color=black, circle, inner sep=0.18\ml}
  ]
  \node[ball color=blue] {}
  child{ node[star] {}
  child{ node {} }
  child{ node[ball color=white] {} } }
  ;
}

\newcommand{\tv}{
  \tikz [
    grow=up,
    level distance=1.0\ml,
    sibling distance=1.0\ml,	
    every  node/.style={shading=ball, ball color=black, circle, inner sep=0.18\ml}
  ]
  \node[ball color=red, star] {}
  child{ node {} }
  child{ node[ball color=white] {}
  child{ node[star] {} }}
  ;
}

\newcommand{\tuv}{
  \tikz [
    grow=up, 	
    level distance=1.0\ml,
    level 1/.style={sibling distance=1.8\ml},
    level 2/.style={sibling distance=1.0\ml},
    every node/.style={shading=ball, ball color=black, circle, inner sep=0.18\ml}
  ]
  \node[ball color=blue] {}
  child{ node[ball color=red, star] {}
  child{ node {} }
  child{ node[ball color=white] {}
  child{ node[star] {} }}}
  child{ node[star] {}
    child{ node {} }
    child{ node[ball color=white] {} }} ;
  }

  \newcommand{\tvu}{
    \tikz [
      grow=up,
      level distance=1.0\ml,
     sibling distance=1.0\ml,
      every node/.style={shading=ball, ball color=black, circle, inner sep=0.18\ml}
    ]
    \node[ball color=red, star] {}
    child{ node {} }
    child{ node[ball color=white] {}
    child{ node[star] {} }}
    child{ node[ball color=blue] {}
	child{ node[star] {}
    child{ node {} }
    child{ node[ball color=white] {} }}}
    ;
  }

  \newcommand{\tuvr}{
  \tikz [
      grow=up,
      level distance=1.0\ml,
      sibling distance=1.0\ml,
      every node/.style={shading=ball, ball color=black, circle, inner sep=0.18\ml}
    ]
    \node[ball color=white] {}
    child{ node[star]{} }
    child{ node[ball color=red, star] {}
    child{ node {}}
    child{ node[ball color=blue] {}
	child{ node[star] {}
    child{ node {}}
	child{ node[ball color=white] {}}}}}
    ;
  }

  \newcommand{\tuvl}{
  \tikz [
      grow=up,
      level distance=1.0\ml,
      sibling distance=1.0\ml,
      every node/.style={shading=ball, ball color=black, circle, inner sep=0.18\ml}
    ]
	\node[star] {}
    child{ node[ball color=blue] {}
    child{ node[ball color=red, star] {}
    child{ node {}}
    child{ node[ball color=white] {}
    child{ node[star] {}}}}}
    child{ node {}}
    child{ node[ball color=white] {}}
    ;
  }

  \newcommand{\trootless}{
    \tikz [
      level distance=1.0\ml,
      sibling distance=1.4\ml,
      every node/.style={shading=ball, ball color=black, circle, inner sep=0.18\ml}
    ]
    \node[ball color=blue] {}
    child[grow=right]{ node[ball color=red, star] {}
    child{ node {} }
    child{ node[ball color=white] {}
    child{ node[star] {} }}}
	child[grow=left]{ node[star] {}
    child{ node {} }
    child{ node[ball color=white] {} }}
    ;
  }

  \newcommand{\maketree}[1]{
    \tikz [
      grow=up, 				
      level distance=1.3\ml,
      sibling distance=2.7\ml,	
      every node/.style={circle, inner sep=0.05\ml},
      baseline=(current bounding box.base)
    ]
	#1
    ;
  }
  \newcommand{\dr}{
	  \rn{black}{$_{q_1}\hspace{0.8\ml}$}
  }
  \newcommand{\dn}[1]{
	  \n{black}{$_{q_#1}\hspace{0.8\ml}$}
  }
  \newcommand{\sr}{
	  \rn{white}{$_{q_1,m_1}$}
  }
  \newcommand{\sn}[1]{
	  \n{white}{$_{q_#1,m_#1}$}
  }
  \newcommand{\n}[2]{
	  node[circle, ball color={#1}, inner sep=0.18\ml, label={[yshift=-0\ml]right:{#2}},label={[yshift=+0\ml]left:{\phantom{#2}}}] {}
  }
  \newcommand{\rn}[2]{
	  \node[circle, ball color={#1}, inner sep=0.18\ml, label={[yshift=-0\ml]right:{#2}},label={[yshift=+0\ml]left:{{\phantom{#2}}}}] {}
  }

  \newcommand{\texk}{
    \maketree{\rn{black}{$_{2}$}
    child{\n{white}{$_{3,2}$}
    child[grow=45,level distance=1.85\ml]{\n{white}{$_{1,2}$}}}
    child{\n{white}{$_{2,1}$}
    child{\n{black}{$_{1}$}}
    child{\n{black}{$_{1}$}}}
  }}

\title[General order conditions for SPRK methods]{General order conditions for stochastic partitioned Runge--Kutta methods}
\author{Sverre Anmarkrud}
\address{Faculty of Environmental Sciences and Natural Resource Management, Norwegian University of Life Sciences, 1433 {\AA}s, Norway}
\address{Department of Mathematics and Computer Science, University of Southern Denmark, 5230 Odense M, Denmark}
\address{Department of Mathematical Sciences, Norwegian University of Science and Technology,
7491 Trondheim, Norway}

\author{Kristian Debrabant}
\author{Anne Kv{\ae}rn{\o}}
\newcommand{\partitionnum}{Q}
\newcommand{\partitionindex}{q}
\newcommand{\diffusionnum}{M}
\newcommand{\diffusionindex}{m}
\newcommand{\equaltreeindex}{R}
\newcommand{\treeindex}{k}
\newcommand{\dop}{\mathrm{d}}
\newcommand{\dW}{\dop W}
\newcommand{\ds}{\dop s}
\newcommand{\dt}{\dop t}

\newcommand{\dX}{\dop X}
\newcommand{\hatTS}{\widehat{TS}}

\DeclareMathOperator{\Sigm}{Sigm}
\begin{document}
\keywords{Stochastic differential equation, partitioned stochastic Runge--Kutta methods, stochastic B-series}

\subjclass{65C30, 60H35, 65C20}

\begin{abstract}
In this paper stochastic partitioned Runge--Kutta (SPRK) methods are considered. A general order
theory for SPRK methods based on stochastic B-series and multicolored, multishaped rooted trees is
developed. The theory is applied to prove the order of some known methods, and it is shown how the
number of order conditions can be reduced in some special cases, especially that the
conditions for preserving quadratic invariants can be used as simplifying assumptions.
\end{abstract}
\maketitle

\section{Introduction}
Whenever the right hand side of an ordinary differential equation (ODE) can be split into different
parts with different features, it might be worth trying to solve the different parts by different
methods. Distinguished examples are stiff/nonstiff parts solved by implicit/explicit methods, fast/slow parts
solved by multirate methods, or partitioned symplectic methods for separable Hamiltonian systems.
The latter seems to have been the aim when partitioned methods for stochastic differential equations (SDEs) are constructed, e.g.\
\cite{milstein02sio, burrage09asd, ma12sca, hong15poq, ma15ssp, holm16sdh, wang17cos}.
Although higher order methods have been constructed, to our knowledge there is still no
complete order theory for stochastic partitioned Runge--Kutta (SPRK) methods. Such an order theory based on
multicolored, multishaped rooted trees is the main contribution in this paper. {The theory is
valid for both It\^o and Stratonovich SDEs\label{rev:21}}.

Constructing high order methods for SDEs is a notoriously nontrivial task. One problem is the required
high order stochastic increments. Another is the huge number of order conditions that have to be
fulfilled. For this reason, most methods are constructed for problems with some special structure,
like linear SDEs, SDEs with additive noise, or separable SDEs. We will demonstrate how the order
theory can be simplified in some of these cases. In particular,
Sanz-Serna and Abia \cite{sanzserna91ocf, abia1993partitioned} have proven that for Runge--Kutta and
partitioned Runge--Kutta methods preserving quadratic invariants only order conditions related to
rootless trees have to be satisfied. A similar result has been proved for stochastic Runge--Kutta
methods in \cite{anmarkrud17ocf}. {This theory depends on the product rule valid for Stratonovich integrals and is therefore only applicable for Stratonovich SDEs}. It will be extended to {Stratonovich} SPRK methods here.

In this paper we consider a system of stochastic partitioned differential equations with
$\partitionnum$ partitions and $\diffusionnum$ diffusion terms,
\begin{equation}
  X^{(\partitionindex)}(t)=x_0^{(\partitionindex)}+\sum_{\diffusionindex=0}^\diffusionnum\int_0^tg^{(\partitionindex)}_{\diffusionindex}(X^{(1)}(s),X^{(2)}(s),\dots,X^{(\partitionnum)}(s))\star
  \dW_\diffusionindex(s), \quad \partitionindex=1,\ldots,\partitionnum \label{eq:psde},
\end{equation}
for which we will also use the abbreviated form
\begin{equation}
  \dop X^{(\partitionindex)}(t)=\sum_{\diffusionindex=0}^\diffusionnum g^{(\partitionindex)}_{\diffusionindex}(X^{(1)}(t),X^{(2)}(t),\dots,X^{(\partitionnum)}(t))\star
  \dW_\diffusionindex(t),
  \quad X^{(q)}(0)=x^{(q)}_0. \label{eq:psde_diff}
\end{equation}
To simplify the notation the deterministic terms are represented by $\diffusionindex=0$, such that $\dW_0(s)=\ds$, while $W_\diffusionindex$, $\diffusionindex=1,\dots,\diffusionnum$, denote one-dimensional and pairwise independent Wiener processes.
The integrals w.r.t.\ the Wiener processes are interpreted as either {It\^o} integrals,
$\star \dW_\diffusionindex(s)=\dW_\diffusionindex(s)$, or Stratonovich integrals, $\star \dW_\diffusionindex(s)=\circ \dW_\diffusionindex(s)$.
We also define the vector of initial values, $x_0 = [x^{(1)}_{0}, x^{(2)}_{0}, \dots, x^{(\partitionnum)}_{0}]$.
Furthermore, we assume that the coefficients $g^{(\partitionindex)}_{\diffusionindex}: \mathbb{R}^{d_1}\times \ldots \times \mathbb{R}^{d_\partitionnum} \to \mathbb{R}^{d_q}$ are sufficiently smooth,
and that the conditions of the existence and uniqueness theorem \cite{oksendal03sde} are satisfied.
The systems are considered to be autonomous. Nonautonomous systems can be included by {extending one of the variables $X^{(\partitionindex)}$, $\partitionindex=1,\dots,\partitionnum$, by $t$}, or by considering the equation $t'=1$ as a separate partitioning.

Denote the numerical approximation of $X^{(\partitionindex)}$ at time $t_n$ by $Y_n^{(\partitionindex)}$.
For the solution of \cref{eq:psde} we consider general, $s$-stage {SPRK} methods given by
\begin{subequations}
	\label{eq:psrk}
	\begin{align}
		H^{(\partitionindex)}_i&=Y^{(\partitionindex)}_{n} + \sum_{\diffusionindex=0}^\diffusionnum\sum_{j=1}^{s}Z^{(\partitionindex,\diffusionindex)}_{i,j}g^{(\partitionindex)}_{\diffusionindex}(H^{(1)}_j,\dots,H^{(\partitionnum)}_j),\qquad i=1,\dots,s,
		\label{eq:psrk_a} \\
		Y^{(\partitionindex)}_{n+1} &= Y^{(\partitionindex)}_{n} +
		\sum_{\diffusionindex=0}^\diffusionnum\sum_{i=1}^s\gamma^{(\partitionindex,\diffusionindex)}_i
		g^{(\partitionindex)}_{\diffusionindex}(H^{(1)}_i,\dots,H^{(\partitionnum)}_i),
		\label{eq:psrk_b}
	\end{align}
\end{subequations}
for partitions $\partitionindex=1,2,\ldots,\partitionnum$. The coefficients $\gamma_i^{(\partitionindex,\diffusionindex)}$ and $Z_{i,j}^{(\partitionindex,\diffusionindex)}$, $i,j=1,\dots,s$, include random variables that depend on the stepsize $h$.
 The random variables used in each step are assumed to be i.i.d., and also $h$ might change from step to step.  To simplify the notation, here and in the following we omit to indicate this by an additional index $n$ on $h$, the random variables used, the coefficients ${\gamma^{(\partitionindex,\diffusionindex)} =
  (}\gamma^{(\partitionindex,\diffusionindex)}_i{)_{i=1}^s}\label{rev:23}$
   and ${Z^{(\partitionindex,\diffusionindex)}=(}Z_{i,j}^{(\partitionindex,\diffusionindex)}{)_{i,j=1}^s}$, and the stage values $H^{(\partitionindex)}_i$. This will in particular also hold for the Wiener increments $\Delta W_\diffusionindex=\int_{t_n}^{t_n+h}\dW_\diffusionindex(s)$.

The coefficients of {an SPRK method} can be gathered in a generalized Butcher tableau. {In the frequently encountered case that there exist matrix functions $Z^{(\partitionindex)}$ and vector functions $\gamma^{(\partitionindex)}$ such that with some vectors of random variables $\xi_m$ it holds that $Z^{(\partitionindex,\diffusionindex)}=Z^{(\partitionindex)}(\xi_\diffusionindex)$ and $\gamma^{(\partitionindex,\diffusionindex)}=\gamma^{(\partitionindex)}(\xi_\diffusionindex)$ for $\diffusionindex=1,\dots,\diffusionnum$ and $\partitionindex=1,\dots,\partitionnum$, we will write the Butcher tableau as follows:}
\[\renewcommand{\arraystretch}{1.5}
  \begin{array}{c|c}
  Z^{(1,0)}& Z^{(1,\diffusionindex)}\\\hline
   \vdots&\vdots \\ \hline
    Z^{(\partitionnum,0)}& Z^{(\partitionnum,\diffusionindex)}
    \\ \hline \hline
    (\gamma^{(1,0)})^{\top}&  (\gamma^{(1,\diffusionindex)})^{\top}\\\hline
   \vdots&\vdots \\ \hline
   (\gamma^{(\partitionnum,0)})^\top & (\gamma^{(\partitionnum,\diffusionindex)})^\top
  \end{array},\qquad \diffusionindex=1,\dots,\diffusionnum.
\]
\begin{remark}
The splitting \eqref{eq:psde} is sometimes called a horizontal splitting. The results are equally
valid for a vertical splitting, that is SDEs split by
\begin{equation}
  \hat{X}(t)=\hat{x}_0+\sum_{\partitionindex=1}^{\partitionnum}
  \sum_{\diffusionindex=0}^\diffusionnum\int_0^t
  \hat{g}^{(\partitionindex)}_{\diffusionindex}
  \left(\hat{X}(s)\right)\star
  \dW_\diffusionindex(s), \quad \partitionindex=1,\ldots,\partitionnum, \label{eq:psde_vert}
\end{equation}
as \cref{eq:psde} can, assuming $d_1=\dots=d_\partitionnum$ (extending the $X^{(\partitionindex)}$
by zero components were necessary), be transformed to \cref{eq:psde_vert} by
\[ \hat{X}(t) = \sum_{\partitionindex=1}^{\partitionnum} X^{(\partitionindex)}(t), \qquad
  g^{(\partitionindex)}_{\diffusionindex}(X^{(1)}(t),X^{(2)}(t),\dots,X^{(\partitionnum)}(t))
  = \hat{g}^{(\partitionindex)}_\diffusionindex\left(\sum_{\partitionindex=1}^{\partitionnum}
  X^{(\partitionindex)}(t) \right).
\]
By summing up the expressions of \cref{eq:psrk} the corresponding SPRK method becomes
\begin{subequations}
	\label{eq:psrk_vert}
	\begin{align}
	  \hat{H}_i &=\hat{Y}_{n} +
	  \sum_{\partitionindex=1}^\partitionnum
	  \sum_{\diffusionindex=0}^\diffusionnum\sum_{j=1}^{s}Z^{(\partitionindex,\diffusionindex)}_{i,j}
	  \hat{g}^{(\partitionindex)}_{\diffusionindex}(\hat{H}_j),\qquad i=1,\dots,s,
		\label{eq:psrk_vert_a} \\
		\hat{Y}_{n+1} &= \hat{Y}_{n} +
\sum_{\partitionindex=1}^{\partitionnum}		
\sum_{\diffusionindex=0}^\diffusionnum\sum_{i=1}^{s}\gamma^{(\partitionindex,\diffusionindex)}_i
		\hat{g}^{(\partitionindex)}_{\diffusionindex}(\hat{H}_i),
		\label{eq:psrk_vert_b}
	\end{align}
\end{subequations}
{where $\hat{Y}_{n}$ is the numerical approximation of $\hat{X}(t_n)$.}
\end{remark}

\begin{example} \label{ex:langevin_sv}
Consider the Langevin equation of motion
\begin{subequations}
	\label{eq:langevin}
\begin{align}
	\dot{r} &= v, \\
	\dot{v} &= f(r,t) - \alpha v + \beta(t),
\end{align}
\end{subequations}
which describes the evolution of a particle with {unit mass}, coordinate $r(t)$ and velocity $v(t)$.
The particle is affected by three forces: $f(r,t)$, a friction force $\alpha v$ for a friction
coefficient $\alpha \geq 0$ and thermal white noise $\beta(t)$. Representing the noise term as $\beta W(t)$ with some constant $\beta$, this can be written as a proper SDE, split vertically
\begin{equation} \label{eq:langevin_vertical}
  \begin{pmatrix}\dop R(t) \\ \dop V(t) \end{pmatrix} =
  \underbrace{\begin{pmatrix} V(t) \\ -\alpha V(t) \end{pmatrix}}_{\hat{g}_0^{(1)}} \dt +
  \underbrace{\begin{pmatrix} 0 \\ f(R(t),t) \end{pmatrix}}_{\hat{g}_0^{(2)}} \dt +
  \underbrace{\begin{pmatrix} 0 \\ \beta \end{pmatrix}}_{\hat{g}_1^{(2)}} \dW,
\end{equation}
or horizontally by
\begin{align*}
  \begin{pmatrix} \dop R(t) \\ \dop U_1(t) \end{pmatrix} & =
  \underbrace{\begin{pmatrix} U_1(t) + U_2(t) \\ -\alpha (U_1(t) + U_2(t)) \end{pmatrix}}_{g_0^{(1)}} \dt, \\
  \dop U_2(t) &= \underbrace{f(R(t),t)}_{g_0^{(2)}}\dt + \underbrace{\beta}_{g_1^{(2)}} \dW(t)
\end{align*}
with $V(t)=U_1(t)+U_2(t)$.

Gr{\o}nbech-Jensen and Farago \cite{gronbech13asa} proposed the following scheme to solve the Langevin equation,
\begin{subequations}
	\label{eq:gronbech_jensen}
	\begin{align}
	  R_{n+1} &= R_n + b h V_n + \frac{b h^2}{2} f(R_n,t_n) + \frac{b h}{2} \beta \Delta W, \\
	  V_{n+1} &= aV_n + \frac{h}{2}(af(R_n,t_n) + f(R_{n+1},t_{n+1})) + b \beta \Delta W,
	\end{align}
\end{subequations}
with
\[
  a = \frac{1-\frac{\alpha h}{2}}{1+\frac{\alpha h}{2}}, \qquad
  b = \frac{1}{1+\frac{\alpha h}{2}}.
\]
This can be reformulated equivalently as {an} SPRK method {with} $\hat{H}_1
  =\begin{pmatrix}
    R_n\\V_n
  \end{pmatrix}$ {and}
\begin{gather*}
  \begin{pmatrix}
    \hat{H}_{2,1}\\
    \hat{H}_{2,2}
  \end{pmatrix}
  =\begin{pmatrix}
R_n\\V_n
  \end{pmatrix}+\frac{h}2\begin{pmatrix}
    \hat{H}_{1,2} + \hat{H}_{2,2}\\
    -\alpha (\hat{H}_{1,2} + \hat{H}_{2,2})
  \end{pmatrix}
  +h\begin{pmatrix}
    0\\
    f(\hat{H}_{1,1},t_n)
  \end{pmatrix}
  + \Delta W \begin{pmatrix} 0 \\ \beta \end{pmatrix}, \\
  \begin{multlined}
  \begin{pmatrix}
    R_{n+1} \\ V_{n+1}
  \end{pmatrix}
  =
  \begin{pmatrix}
    R_{n} \\ V_{n}
  \end{pmatrix} +
  \frac{h}2
  \begin{pmatrix}
    \hat{H}_{1,2} + \hat{H}_{2,2}\\
    -\alpha (\hat{H}_{1,2} + \hat{H}_{2,2})
  \end{pmatrix}
  +\frac{h}{2} \begin{pmatrix}
    0 \\
    f(\hat{H}_{1,1},t_n)+f(\hat{H}_{2,1},t_{n+1})
  \end{pmatrix}\\
  + \Delta W \begin{pmatrix} 0 \\ \beta \end{pmatrix},
  \end{multlined}
\end{gather*}
{where $\hat{H}_{i,j}$ denotes component $j$ of $\hat{H}_i$.}
{
Due to the special structure of the problem, the scheme \eqref{eq:gronbech_jensen} can be interpreted
as an application of several SPRK methods (with {$M=1$}),
of which we mention the following two:
\begin{equation}
  \renewcommand{\arraystretch}{1.5}
  \label{eq:sv}
 {\begin{array}{cc|cc}
    0 & 0 & 0 &  0 \\
    \frac{h}{2} & \frac{h}{2} & 0 & 0  \\
    \hline
    0 & 0 &  0 & 0 \\
    h & 0 &  \Delta W_\diffusionindex & 0 \\
    \hline \hline
    \frac{h}{2} & \frac{h}{2} &  0 & 0 \\
    \hline
    \frac{h}{2} & \frac{h}{2} &   \Delta W_\diffusionindex & 0 \\
  \end{array},} \qquad \qquad
  \begin{array}{cc|ccc}
    0 & 0 &  0 &  0 \\
    \frac{h}{2} & \frac{h}{2} & \frac{\Delta W_\diffusionindex}{2} & \frac{\Delta W_\diffusionindex}{2}  \\
    \hline
    0 & 0 &  0 & 0 \\
    h & 0 &  \Delta W_\diffusionindex & 0 \\
    \hline \hline
    \frac{h}{2} & \frac{h}{2} &   \frac{\Delta W_\diffusionindex}{2} & \frac{\Delta W_\diffusionindex}{2} \\
    \hline
    \frac{h}{2} & \frac{h}{2} &   \frac{\Delta W_\diffusionindex}{2} & \frac{\Delta W_\diffusionindex}{2} \\
  \end{array}, \qquad m=1,\dots,M.
\end{equation}
While the method to the left might be the obvious choice, it is only convergent for partitioned
problems for which there is no noise in the first partitioning. In this particular case it coincides
with the method to the right. In \cref{subsec:SDEswithadditivenoise}   we will prove this method to
be of strong order 1 and weak order 2 when applied to {more general} partitioned SDEs with additive noise. }
\end{example}

In \cref{sec:OrderTheory}, a general order theory for SPRKs is developed. The theory is based on stochastic
B-series and multicolored, multishaped rooted trees.
In \cref{sec:ApplicationOfTheory} two particular cases are studied: SDEs with additive noise and separable
problems. We will here present examples of how the B-series theory can be used to find the order of
some given methods. In \cref{sec:quadraticinvariants}, it is shown how the number of order conditions can be further
reduced in the case of methods preserving quadratic invariants {of  Stratonovich SDEs}\label{rev:22}.

\section{Order theory}\label{sec:OrderTheory}
B-series for deterministic ODEs were introduced by J.\ C.\ Butcher \cite{butcher63cft} in 1963.
B-series for SDEs were developed by Burrage and Burrage  \cite{burrage96hso, burrage00oco, burrage99thesis} for strong convergence of Stratonovich SDEs,
by Komori, Mitsui and Sugiura \cite{komori97rta} and Komori \cite{komori07mcr} for weak convergence of Stratonovich SDEs, and by
R\"{o}{\ss}ler \cite{roessler04ste, roessler06rta} for weak convergence in both the {It\^o} and Stratonovich case.
A unified theory for B-series encompassing both weak and strong convergence for both {It\^o} and Stratonovich SDEs was given in \cite{debrabant08bsa}.
In the following we will generalize this to SPRKs.

Our first goal is to find B-series representations of \cref{eq:psrk}, and we begin by assuming $X^{(\partitionindex)}(h)$ can be written as a B-series $B^{(\partitionindex)}(\phi,x_0;h)$,
\begin{align*}
	B^{(\partitionindex)}(\phi,x_0;h)=\sum_{\tau \in T_\partitionindex}\alpha(\tau)\cdot \phi(\tau)(h)\cdot F(\tau)(x_0),
\end{align*}
where $T_\partitionindex$ is the set of shaped, colored, rooted trees as defined below.
The terms $\alpha(\tau)$ are combinatoric terms. The elementary weight functions $\phi(\tau)(h)$ are stochastic integrals or random variables, and $F(\tau)(x_0)$ are the elementary differentials. {To simplify the presentation, we assume that all elementary differentials exist and all considered B-series converge. Otherwise, one has to consider truncated B-series and discuss the remainder term \cite{roessler04ste}.\label{rev1:1}}
\begin{definition}[Trees and combinatorial coefficients]
	The set of shaped, rooted trees
	\begin{align*}
		T =T_1 \cup T_2 \cup \dots \cup T_\partitionnum
	\end{align*}
	where
	\begin{align*}
		T_\partitionindex &= \{\emptyset_\partitionindex\} \cup T_{\partitionindex,0} \cup T_{\partitionindex,1} \cup \dots \cup T_{\partitionindex,\diffusionnum}
	\end{align*}
	for $\partitionindex=1,\ldots,\partitionnum$ is recursively defined as follows:
	    \begin{romannum}
        \item The graph
			$\bullet_{\partitionindex,\diffusionindex}$ with only one vertex of shape $\partitionindex$ and color $\diffusionindex$ belongs to $T_{\partitionindex,\diffusionindex}$.
		\item If $\tau_1,\tau_2,\dots,\tau_{\kappa} \in T\setminus\{\emptyset_1,\dots,\emptyset_{\partitionnum}\}$, then $[\tau_1,\tau_2,\dots,\tau_{\kappa}]_{\partitionindex,\diffusionindex}\in T_{\partitionindex,\diffusionindex}$, where
 $[\tau_1,\tau_2,\dots,\tau_{\kappa}]_{\partitionindex,\diffusionindex}$ denotes the tree formed by joining the subtrees $\tau_1,\tau_2,\dots,\tau_{\kappa}$ each by a single branch to a common root of shape $\partitionindex$ and color $\diffusionindex$.
 \end{romannum}
Further, we define $\alpha(\tau)$ as
\begin{align*}
		\alpha(\emptyset_\partitionindex) = 1, \; \alpha(\bullet_{\partitionindex,\diffusionindex}) = 1, \; \alpha(\tau=[\tau_1, \dots, \tau_{\kappa}]_{\partitionindex,\diffusionindex}) = \frac{1}{r_1!r_2!\dots r_\equaltreeindex!}\prod_{\treeindex=1}^{\kappa}\alpha(\tau_\treeindex),
\end{align*}
where $r_1,r_2,\dots,r_\equaltreeindex$ count equal trees among $\tau_1,\tau_2,\dots,\tau_{\kappa}$.
\end{definition}

\begin{remark}
  Rooted trees can be represented in the bracket notation, as used in the definition, or
  illustrated as graphs (see \cref{fig:example_tree}). To ease the reading, deterministic nodes
  are in the latter represented as black nodes, with the $m=0$ omitted, while stochastic nodes are
  white. \Cref{fig:example_tree} also gives the corresponding examples for the functions of
  trees that will be defined in the following.
\end{remark}
\begin{figure}[t]
       \begin{minipage}{0.33\linewidth}
	\begin{tabular}{c}
	  $\tau$ \\[1em] \texk \\[2em]
	$[[\bullet_{1,0},\bullet_{1,0}]_{2,1}],[\bullet_{1,2}]_{3,2}]_{2,0}$
	\end{tabular}
      \end{minipage}
      \begin{minipage}{0.64\linewidth}
	\begin{align*}
 	  \alpha(\tau) &= \frac{1}{2} \\
	  F(\tau) &= (D_{2,3}g_{0}^{(2)})((D_{11}g_{1}^{(2)})(g_0^{(1)},g_0^{(1)}),(D_1g_2^{(3)})g_2^{(1)})
	  \\
	  \phi(\tau) &= \int_0^h \left(\int_0^s \! s_1^2 \dW_1(s_1)\right) \left( \int_0^{s} \!
	W_{{2}}(s_2)\dW_{{2}}(s_2) \right) \dop s \\
	  \Phi(\tau) &= \sum_{i,j,k,l,m,n=1}^s
	  \gamma_i^{(2,0)}Z_{i,j}^{(2,1)}Z_{j,k}^{(1,0)}Z_{j,l}^{(1,0)}Z_{i,m}^{(3,2)}Z_{{m},n}^{(1,2)} \\
	  \rho(\tau) &=\frac92
	\end{align*}
      \end{minipage}
\caption{An example of a shaped, colored, rooted tree and its corresponding functions.}
\label{fig:example_tree}
\end{figure}

\begin{definition}[Elementary differentials] For a tree $\tau \in T$ the elementary differential is a mapping $F(\tau)$: $\mathbb{R}^{d_1}\times \ldots \times \mathbb{R}^{d_\partitionnum} \to \mathbb{R}^{d}$ defined recursively by
    \begin{romannum}
		\item $F(\emptyset_\partitionindex)(x_0) = x_{0}^{(\partitionindex)}$, $\emptyset_\partitionindex \in T_\partitionindex$,
		\item $F(\bullet_{\partitionindex,\diffusionindex})(x_0)=g^{(\partitionindex)}_{\diffusionindex}(x_0)$,
		\item If $\tau = [\tau_1,\tau_2,\dots,\tau_{\kappa}]_{\partitionindex,\diffusionindex}\in T_{\partitionindex,\diffusionindex}$, then
            \begin{equation*}
				F(\tau)(x_0)=(D_{\partitionindex_1\dots\partitionindex_\kappa}g^{(\partitionindex)}_{\diffusionindex})(x_0)(F(\tau_1)(x_0),F(\tau_2)(x_0),\dots,F(\tau_{\kappa})(x_0))
            \end{equation*}
            where $\partitionindex_\treeindex$ is the shape of
	    $\tau_\treeindex$, $\treeindex=1,\dots,\kappa$, and
	    $D_{\partitionindex_1\dots\partitionindex_\kappa}=\frac{\partial^{\kappa}}{\partial x^{\partitionindex_1}\dots\partial x^{\partitionindex_\kappa}}$ denotes the derivative operator of order $\kappa$.
    \end{romannum}
\end{definition}
Fundamental for this work is the following lemma which says that if $Y^{(\partitionindex)}(h)$ can be written as a B-series, then $f(Y^{(1)}(h),\dots,Y^{(\partitionnum)}(h))$ can also be written as a B-series.
This is a trivial extension of the lemma found in \cite{debrabant08bsa}.
\begin{lemma}\label{lem:function_of_b-series}
	If $Y^{(\partitionindex)}(h) = B^{(\partitionindex)}(\phi,x_0;h)$, $\partitionindex=1,\dots,\partitionnum$, are some B-series and $f \in C^{\infty}(\mathbb{R}^{d_1}\times \ldots \times \mathbb{R}^{d_\partitionnum}, \mathbb{R}^{d})$, then $f(Y^{(1)}(h),\dots,Y^{(\partitionnum)}(h))$ can be written as a formal series of the form
    \begin{equation}
		f(Y^{(1)}(h),\dots,Y^{(\partitionnum)}(h)) = \sum_{u\in U_f} \beta(u)\cdot \psi_{\phi}(u)(h)\cdot G(u)(x_0),
    \end{equation}
    where
    \begin{romannum}
         \item $U_f$ is a set of trees derived from T, by $\bullet_f \in U_f$, and if $\tau_1,\tau_2,\dots,\tau_{\kappa} \in T\setminus\{\emptyset_1,\dots,\emptyset_{\partitionnum}\}$, then $[\tau_1,\tau_2,\dots,\tau_{\kappa}]_f \in U_f$.
        \item $G(\bullet_f)(x_0)=f(x_0)$ and
	  \[G([\tau_1,\tau_2,\dots,\tau_{\kappa}]_f)(x_0)=({D_{q_1,\dotsc,q_{\kappa}}} f)(x_0)(F(\tau_1)(x_0),\dots,F(\tau_{\kappa})(x_0)).\]
%{where $\partitionindex_\treeindex$ is the shape of $\tau_\treeindex$, $\treeindex=1,\dots,\kappa$.}
        \item $\beta(\bullet_f)=1$ and \[\beta([\tau_1,\tau_2,\dots,\tau_{\kappa}]_f)=\frac{1}{r_1!r_2!\cdots r_q!}\prod_{\treeindex=1}^{\kappa}\alpha(\tau_\treeindex),\] where $r_1,r_2,\dots,r_q$ count equal trees among $\tau_1,\tau_2,\dots,\tau_{\kappa}$.
		\item $\psi_{\phi}(\bullet_f)(h){=} 1$ and  $\psi_{\phi}([\tau_1,\tau_2,\dots,\tau_{\kappa}]_f)(h)=\prod_{\treeindex=1}^{\kappa}\phi{(\tau_\treeindex)}(h)$.
    \end{romannum}
\end{lemma}
If we apply \cref{lem:function_of_b-series} to the functions $g^{(\partitionindex)}_{\diffusionindex}$ in \cref{eq:psde} we get
\begin{align*}
	g^{(\partitionindex)}_{\diffusionindex}(X^{(1)}(h),\dots,X^{(\partitionnum)}(h))=\sum_{u \in U_{g^{(\partitionindex)}_{\diffusionindex}}}\beta(u)\cdot \psi_{\phi}(u)(h)\cdot G(u)(x_0).
\end{align*}
By the definitions of trees, $T_{\partitionindex,\diffusionindex}$, and elementary differentials, $F(\tau)(x_0)$, we can write this as
\begin{align}
	\label{eq:part_funs}
	g^{(\partitionindex)}_{\diffusionindex}(X^{(1)}(h),\dots,X^{(\partitionnum)}(h))=\sum_{\tau \in T_{\partitionindex,\diffusionindex}}\alpha(\tau) \cdot \phi'_{\partitionindex,\diffusionindex}(\tau)(h)\cdot F(\tau)(x_0),
\end{align}
where
\begin{align*}
	\phi_{\partitionindex,\diffusionindex}'(\tau)(h) &= \begin{cases} 1,& \text{ if } \tau = \bullet_{\partitionindex,\diffusionindex},\\
		\prod_{\treeindex=1}^\kappa \phi(\tau_\treeindex)(h),& \text{ if } \tau = [\tau_1,\dots,\tau_\kappa ]_{\partitionindex,\diffusionindex} \in T_{\partitionindex,\diffusionindex}.\end{cases}
\end{align*}

We now write the exact solutions {of} \cref{eq:psde} as B-series and use \cref{eq:part_funs} to obtain
\begin{multline*}
	\sum_{\tau \in T_\partitionindex} \alpha(\tau)\cdot \phi(\tau)(h) \cdot F(\tau)(x_0) = \\
	x_{0}^{(\partitionindex)} + \sum_{\diffusionindex=0}^\diffusionnum\sum_{\tau \in T_{\partitionindex,\diffusionindex}} \alpha(\tau)\cdot \int_0^h\phi_{\partitionindex,\diffusionindex}'(\tau)(s)\star \dW_\diffusionindex(s)F(\tau)(x_0).
\end{multline*}
Comparing term by term we see that
\begin{multline*}
	\phi(\emptyset_\partitionindex)(h){=} 1, \text{ and } \phi(\tau)(h)=\int_0^h\phi_{\partitionindex,\diffusionindex}'(\tau)(s) \star \dW_\diffusionindex(s) \\
	\text{ for all } \tau \in T_{\partitionindex,\diffusionindex}, \; \diffusionindex=0,1,\dots,\diffusionnum,  \partitionindex=1,\dots,\partitionnum.
\end{multline*}
With induction on {the height of} $\tau$ we have proven the following theorem.
\begin{theorem}The exact solutions $X^{(\partitionindex)}({h})$ of \cref{eq:psde}, {$\partitionindex=1,\dots,\partitionnum$}, can be written as B-series $B^{(\partitionindex)}(\phi,x_0;h)$ with
	\begin{align*}
		&\phi(\emptyset_\partitionindex)(h){=} 1, \quad \phi(\bullet_{\partitionindex,\diffusionindex})(h)={\Delta} W_\diffusionindex(h), \\
		\phi([\tau_1,\tau_2,&\dots,\tau_\kappa]_{\partitionindex,\diffusionindex})(h)=\int_0^h\prod_{j=1}^\kappa\phi(\tau_j(s)) \star \dW_\diffusionindex(s), \\
		&\text{for all }[\tau_1,\tau_2,\dots,\tau_\kappa]_{\partitionindex,\diffusionindex} \in T_{\partitionindex,\diffusionindex}, \;  \partitionindex=1,\dots,\partitionnum, \; \diffusionindex=0,1,\dots,\diffusionnum.
	\end{align*}
	\label{thm:part_exact}
\end{theorem}

A similar result can be found for the numerical solution of \cref{eq:psde} by the $s$-stage SPRK method \eqref{eq:psrk}.
\begin{theorem}
	\label{thm:part_srk_b-series}
	The numerical solutions $Y^{(\partitionindex)}_{1}$ as well as the stage values can be written in terms of B-series
	\begin{align*}
		H^{(\partitionindex)}_i&=B^{(\partitionindex)}(\Psi_i,x_0;h), \quad Y^{(\partitionindex)}_{1}=B^{(\partitionindex)}(\Phi,x_0;h)
	\end{align*}
	for all $i=1,\dots,s$, $\partitionindex=1,\ldots,\partitionnum$, with
	\begin{subequations}
		\label{eq:num_b_h}
		\begin{align}
			\Psi_i(\emptyset_\partitionindex)(h) &= 1, \quad \Psi_i(\bullet_{\partitionindex,\diffusionindex})(h)=\sum_{j=1}^sZ^{(\partitionindex,\diffusionindex)}_{i,j}, \label{eq:num_b_h_a} \\
			\Psi_i([\tau_1,&\dots,\tau_\kappa]_{\partitionindex,\diffusionindex})(h)=\sum_{j=1}^sZ^{(\partitionindex,\diffusionindex)}_{i,j}\prod_{\treeindex=1}^{\kappa}\Psi_j(\tau_\treeindex)(h) \label{eq:num_b_h_b}
		\end{align}
	\end{subequations}
and
\begin{subequations}
	\label{eq:num_b_y}
	\begin{align}
		\Phi(\emptyset_\partitionindex)(h)&= 1, \quad \Phi(\bullet_{\partitionindex,\diffusionindex})(h)=\sum_{i=1}^{s}\gamma^{(\partitionindex,\diffusionindex)}_i\label{eq:num_b_y_a},\\
		\Phi(\tau=[\tau_1,&\dots,\tau_\kappa]_{\partitionindex,\diffusionindex})(h)=\sum_{i=1}^{s}\gamma^{(\partitionindex,\diffusionindex)}_i\prod_{\treeindex=1}^{\kappa}\Psi_i(\tau_\treeindex)(h). \label{eq:num_b_y_b}
	\end{align}
\end{subequations}

\end{theorem}
\qquad \begin{proof}
	Write $H_i^{(\partitionindex)}$ as B-series,
\begin{align*}
	H^{(\partitionindex)}_i&=\sum_{\tau \in T_\partitionindex}\alpha(\tau)\Psi_i(\tau)(h)F(\tau)(x_0),
\end{align*}
for $i=1,\dots,s$, $\partitionindex=1,\dots,\partitionnum$, where as usual the product of vectors is understood componentwise.
Inserted into \cref{eq:psrk_a} and using \cref{eq:part_funs} this gives
\begin{align*}
	H_i^{(\partitionindex)}&= x_{0}^{(q)}+\sum_{\diffusionindex=0}^\diffusionnum\sum_{j=1}^{s}\sum_{\tau \in T_{\partitionindex,\diffusionindex}}\alpha(\tau)\left(Z^{(\partitionindex,\diffusionindex)}_{i,j}\cdot \Psi_{j,\partitionindex,\diffusionindex}'(\tau)(h)\right)F(\tau)(x_0).
\end{align*}
A term by term comparison yields \cref{eq:num_b_h}. The proof of \cref{eq:num_b_y} is similar.
\qquad \end{proof}

The local order of accuracy of the SPRK method can now be decided by comparing the B-series of the exact and the numerical solution. {\label{rev:11}Applying a time transformation, it is hereby sufficient to consider B-series expansions around $t=0$, as done in \cref{thm:part_exact,thm:part_srk_b-series}.} First, we need to define the tree order.
{
\begin{definition}\label{def:treeorder}
	The order of a tree $\tau \in T$ is defined by
	\begin{gather*}
		\rho(\emptyset_\partitionindex)=0,\qquad \rho(\tau=[\tau_1,\dots,\tau_\kappa]_{\partitionindex,\diffusionindex})=\sum_{\treeindex=1}^\kappa
		\rho(\tau_\treeindex)+ \begin{cases} 1  \text{   for } \diffusionindex=0, \\
		  \frac{1}{2} \text{   otherwise{.}}\end{cases}
	\end{gather*}
\end{definition}}
The following theorem relates the global order of accuracy to the local order. Here, we assume that
method \eqref{eq:psrk} is constructed such that $\Phi(\tau)(h) = \mathcal{O}(h^{\rho(\tau)})$ for all $\tau\in T$ for mean square convergence, respectively
$
E\prod_{k=1}^{\kappa} \Phi(\tau_k) = \mathcal{O}(h^{\sum_{k=1}^{\kappa}\rho(\tau_k)})$ for all $\tau_1,\dots,\tau_\kappa\in T$, $\kappa\in\mathbb{N}$, for weak convergence.
\begin{theorem} \label{thm:milstein}
  The method has mean square global order $p$ if
  \begin{subequations} \label{eq:msord}
    \begin{align}
      \Phi(\tau)(h) &= {\phi}(\tau)(h) +
      \mathcal{O}(h^{p+\frac{1}{2}}), & \forall \tau & \in T \text{ with }
      \rho(\tau)\leq p, \label{eq:msorda} \\
      E\Phi(\tau)(h) &= E {\phi}(\tau)(h) + \mathcal{O}(h^{p+1}), &\forall \tau
      & \in T \text{ with }
      \rho(\tau)\leq p+\frac{1}{2}, \label{eq:msordb}
    \end{align}
  \end{subequations}
  and weak consistency of order $p$ if and only if
  \begin{equation}
    E\prod_{k=1}^{\kappa} \Phi(\tau_k) = E\prod_{k=1}^{\kappa} {\phi}(\tau_k)
    + \mathcal{O}(h^{p+1})
    \quad \text{whenever}\quad \sum_{k=1}^{\kappa}\rho(\tau_k)\leq p+\frac{1}{2}.
    \label{eq:weakord}
  \end{equation}
\end{theorem}
Here, the $\mathcal{O}(\cdot)$-notation refers to $h\to0$ and, especially in \cref{eq:msorda}, to the $L^2$-norm.
The result \eqref{eq:msord} was first proved in \cite{burrage00oco}, while \cref{eq:weakord} is a consequence of a result of Milstein \cite{milstein95nio}, see \cite{debrabant08bsa}
for details. A list of all trees with $\rho(\tau)\leq 2$ and their corresponding functions are given
in \cref{appendix}.
\section{Two special cases}\label{sec:ApplicationOfTheory}
The amount of order conditions to be satisfied for higher order methods is quite overwhelming.
E.\,g.\ for a method of strong order 1.5 with $Q=2$ and scalar noise, order conditions for {122}
different trees need to be satisfied (some of them being trivially fulfilled, though).

The virtue of  partitioned methods becomes clear when applied to problems for which some underlying
structure can be exploited, in the sense that many elementary differentials will be zero, and the
corresponding trees can thus be ignored in the B-series. The main task is to identify those trees.
In this subsection, the idea will be demonstrated with two distinguished examples from Milstein et.\ al
\cite{milstein02nmf,milstein02sio,milstein04snf}, SDEs with additive noise and
separable systems, both with two partitionings.

\subsection{SDEs with additive noise}\label{subsec:SDEswithadditivenoise}

\begin{table}[t]
  \caption{Relevant trees up to order two and corresponding functions for two-partitioned SDEs with additive noise
    \eqref{eq:sde_additive_noise}. Here, $\partitionindex_1,\partitionindex_2,\partitionindex_3\in\{1,2\}$, while $\tilde{\partitionindex}_1\in\{1,2,3\}$. The weights $\Phi(\tau)$ correspond to method \eqref{eq:sv} extended by $Z^{(3,0)}=Z^{(2,0)}$ and
$\gamma^{(3,0)} = \gamma^{(2,0)}$.
  }
  \label{tab:additive}
  \[
  \renewcommand{\arraystretch}{1.5}
  \begin{array}{c|@{}c|c|c|c} \hline
   \!\! \text{No} \!\!&  \tau & \!\rho(\tau)\! & \phi(\tau) & \Phi(\tau)  \\ \hline
   1 & \maketree{\rn{white}{$_{q_1,m_{{1}}}$}}   &\frac{1}{2} & \Delta W_{m_{{1}}} & \Delta W_{m_{{1}}}   \\
   2 &  \maketree{\rn{black}{$_{\tilde{q}_1}$}}   & 1 & h &  h   \\
   6 & \maketree{\rn{white}{$_{q_1,m_{{1}}}$} child{\n{black}{$_3$}}}   & \frac{3}{2} & \int\limits_0^h s\dW_{m_{{1}}}(s) & \frac{1}{2} h \Delta W_{m_{{1}}} \\
   7 &\maketree{\rn{black}{$_{q_1}$} child{\n{white}{$_{q_2,m_{{2}}}$}}}   & \frac{3}{2} & \int\limits_0^h W_{m_{{2}}}(s) \dop s & \frac{1}{2} h\Delta W_{m_{{2}}} \\
   8 &\maketree{\rn{black}{$_{q_1}$} child{\n{black}{$_{\tilde{q}_1}$}}} & 2 & \frac{1}{2} h^2 &
    \frac{1}{2}h^2 \\
   9 & \maketree{\rn{black}{$_{q_1}$} child{\n{white}{$_{q_3,m_3}$}} child{\n{white}{$_{q_2,m_2}$}}}
   & 2 & \int\limits_0^h W_{m_2}(s)W_{m_3}(s)\dop s & \frac{1}{2} h \Delta W_{m_2} \Delta W_{m_3}
       \\ \hline
  \end{array}
\]
\end{table}
We consider partitioned problems with $Q=2$ with additive noise, thus
\begin{subequations}
  \begin{align}
    \dop X^{(1)} &= g^{(1)}_0(X^{(1)},X^{(2)},t)\dt + \sum_{m=1}^\diffusionnum g^{(1)}_m(t) \dW_m(t), \\
    \dop X^{(2)} &= g^{(2)}_0(X^{(1)},X^{(2)},t)\dt + \sum_{m=1}^\diffusionnum g^{(2)}_m(t) \dW_m(t).
  \end{align}
  \label{eq:sde_additive_noise}
\end{subequations}
To deal with the nonautonomous case, let us include a third partition:
\[ \dop X^{(3)} = 1\dt. \]
The problem structure induces that many elementary differentials will be zero, and the corresponding trees can thus be ignored in the B-series. Restricting to nonvanishing elementary differentials, we only need to consider trees characterized by the following properties:
\smallskip
\begin{itemize}
  \item There are no nodes $\bullet_{3,m}$ for $m\not= 0$, since the corresponding $g^{(3)}_m=0$.
  \item Nodes $\bullet_{3,0}$ have no branches, since $g^{(3)}_0$ is constant.
  \item The stochastic nodes $\bullet_{q,m}$ with $q=1,2$ and $m\not=0$ can only be followed by one
    or more nodes $\bullet_{3,0}$, since $g^{(q)}_m$ is only time-dependent.
\end{itemize}
\smallskip
Of the trees listed in the appendix, this leaves us with the trees no.\ 1, 2, 6, 7, 8 and 9 and even those
can be simplified since $q_1\not = 3$ for the trees no.\ 6, 7, 8 and 9, and $q_2=3$ for tree no.\ 6. The
trees together with their corresponding functions are listed in
\cref{tab:additive}.

Let the SDE \eqref{eq:sde_additive_noise} be solved by method \eqref{eq:sv}. For the third partition, we use $Z^{(3,0)}=Z^{(2,0)}$ and
$\gamma^{(3,0)} = \gamma^{(2,0)}$. The corresponding weights $\Phi(\tau)$ are given in \cref{tab:additive}. From \cref{thm:milstein} we can conclude that
the method is of strong order 1, as the order conditions for the trees
no.\ 6 and 7 are only fulfilled in expectation, and that it is of weak order 2.

The  Langevin equation \eqref{eq:langevin} of \cref{ex:langevin_sv} is an example of an SDE \eqref{eq:sde_additive_noise},
thus we can conclude that the method proposed in \cite{gronbech13asa} is of strong order 1 and weak order 2.
{A further example of an SDE \eqref{eq:sde_additive_noise} is the following:
\begin{example}[Stochastic version of the Jansen and Rit Neural Mass Model]\label{ex:JansenRitNeuralMassModell}
Let $X_0$, $X_1$ and $X_2$ describe the mean postsynaptic potentials of the main neural population, the excitatory and inhibitory interneurons, respectively. The following model has been proposed in  \cite{ableidinger17asv}:
\begin{align*}
  \dX_i(t) &= X_{i+3}(t),\qquad i\in\{0,1,2\},\\
  \dX_3(t) &=\left[Aa\left[\mu_3(t)+\Sigm(X_1(t)-X_2(t))\right]-2aX_3(t)-a^2X_0(t)\right]+\sigma_3(t)\dW_3(t),\\
  \dX_4(t) &=\left[Aa\left[\mu_4(t)+C_2\Sigm(C_1X_0(t))\right]-2aX_4(t)-a^2X_1(t)\right]+\sigma_4(t)\dW_4(t),\\
  \dX_5(t) &=\left[Bb\left[\mu_5(t)+C_4\Sigm(C_3X_0(t))\right]-2bX_5(t)-b^2X_2(t)\right]+\sigma_5(t)\dW_5(t),
\end{align*}
where $\Sigm(x)=\frac{\nu_{\max}}{1+e^{r(\nu_0-x)}}$, $\mu_i$ and $\sigma_i$, $i\in\{3,4,5\}$, describe external input and the scaling of the stochastic components, respectively, and $A$, $B$, $a$, $b$, $C_i$, $i\in\{1,2,3,4\}$, $r$ and $\nu_0$, $\nu_{\max}$ are parameters.
\end{example}}

\subsection{Separable systems}
Consider the following system, which typically can arise from separable Hamiltonian systems,
discussed e.g.\ in \cite[4.2.2]{milstein04snf}:
\begin{subequations}\label{rev:eq1}
  \begin{align}
    \dop X^{(1)}(t) &= g_0^{(1)}(X^{(2)}(t))\dt +
                 \sum_{m=1}^M g_m^{(1)}(X^{(2)}(t)) \star \dW_m(t), \\
    \dop X^{(2)}(t) &= g_0^{(2)}(X^{(1)}(t)) \dt.
    \end{align}
  \label{eq:separable}
\end{subequations}
The elementary differentials corresponding to the following trees vanish:
\smallskip
\begin{itemize}
  \item All trees for which a node of shape $q$ {($q\in\{1,2\}$)} is followed by a node of the same shape, as $\frac{\partial g^{(q)}_m}{\partial x^{(q)}}=0$ {for $m=0,\dots,M$}.
  \item All trees with nodes $\bullet_{2,m}$ with $m\not=0$: There is no noise in the second
    partition, so $g_m^{(2)}=0$ {for $m=1,\dots,M$}.
\end{itemize}
\smallskip
The remaining trees {$\tau$} with  $\rho(\tau)\leq 2$ are then (for $q_1,q_2\in\{1,2\}$ with $q_1\neq q_2$)
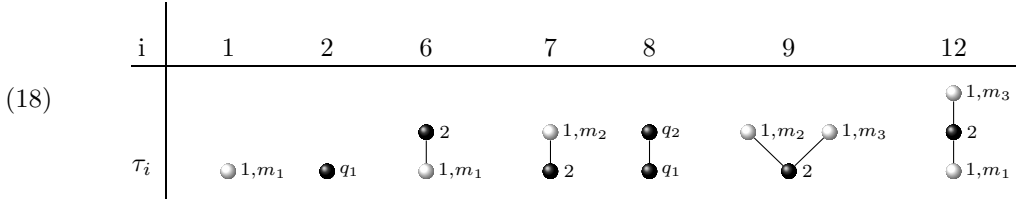
\begin{equation} \label{eq:trees_separable}
  \begin{array}{@{}c|@{}c@{}c@{}c@{}c@{}c@{}c@{}c@{}} \\
    \text{i}&  1 & 2 & 6 & 7 & 8 & 9 & 12 \\ \hline
  \tau_i & \maketree{\rn{white}{$_{1,m_{{1}}}$}}&
  \maketree{\rn{black}{$_{q_1}$}}&
  \maketree{\rn{white}{$_{1,m_{{1}}}$} child{\n{black}{$_2$}}} &
  \maketree{\rn{black}{$_2$} child{\n{white}{$_{1,m_{{2}}}$}}}&
  \maketree{\rn{black}{$_{q_1}$} child{\n{black}{$_{q_2}$}}} &
  \maketree{\rn{black}{$_2$} child{\n{white}{$_{1,m_3}$}} child{\n{white}{$_{1,m_2}$}}}&
  \maketree{\rn{white}{$_{1,m_1}$} child{\n{black}{$_2$} child{\n{white}{$_{1,m_3}$}}}}
\end{array}
\end{equation}
Consider the following method, proposed in
\cite{milstein02sio}:
\begin{equation}
  \renewcommand{\arraystretch}{1.5}
  \begin{array}{cc|cc}
    \frac{h}{4} & 0 & \frac{3J_{(m,0)}}{2h}-\frac{\Delta W_m}{2} & 0 \\
    \frac{h}{4} & \frac{3h}{4} &  \frac{3J_{(m,0)}}{2h}-\frac{\Delta W_m}{2} &
    -\frac{3J_{(m,0)}}{2h}+\frac{3\Delta W_m}{2}
    \\ \hline
    0 & 0 && \\
    \frac{2h}{3} & 0 & & \\ \hline \hline
    \frac{h}{4} & \frac{3h}{4} &  \frac{3J_{(m,0)}}{2h}-\frac{\Delta W_m}{2} &
    -\frac{3J_{(m,0)}}{2h}+\frac{3\Delta W_m}{2}
     \\ \hline
    \frac{2h}{3} & \frac{h}{3} & &
  \end{array}
  \label{eq:milstein_p_oneandahalf}
\end{equation}
where $J_{(m,0)} = \int_{t_n}^{t_n+h} (W_m(s)-W_m(t_n))\ds$.
It is straightforward to show that $\phi(\tau) = \Phi(\tau)$ for {trees $\tau_1,\tau_2,\tau_7,\tau_8$} in
\cref{eq:trees_separable}. Using the fact (see e.\,g.\ \cite{kloeden92nso,debrabant10ste}) that $J_{(m,0)}+J_{(0,m)}=h\Delta W_m$, where
$J_{(0,m)}=\int_{t_n}^{t_n+h}(s-t_n)\star \dW_m(s)$, {\label{rev12}we obtain for $\tau_6$ that (remember that for the consistency analysis it is enough to consider the first step of the method, so $t_n=t_0=0$)
\begin{align*}
\Phi(\tau_6)&= \sum_{i=1}^2 \gamma^{(1,m_1)}_i\sum_{j=1}^2Z^{(2,0)}_{i,j}= \gamma^{(1,m_1)}_2 \sum_{j=1}^2 Z^{(2,0)}_{2,j}\\ &= \left(- \frac{3J_{(m_1,0)}}{2h} + \frac{3\Delta W_{m_1}}{2} \right)\frac{ 2h}3 = - J_{(m_1,0)} + \Delta W_{m_1} h\\
&= J_{(0,m_1)}=\phi(\tau_6).
\end{align*}
}
For the remaining two
trees we get:
\begin{align*}
  \Phi(\tau_9) & = \frac{1}{2}\left(\frac{3}{h}J_{(m_2,0)}J_{(m_3,0)} - \Delta W_{m_2}J_{(m_3,0)} -
  \Delta W_{m_3}J_{(m_2,0)}
+ h\Delta W_{m_2} \Delta W_{m_3} \right), \\
\Phi(\tau_{12}) & = \frac{1}{2}\left(-\frac{3}{h}J_{(m_1,0)}J_{(m_3,0)} + 3\Delta W_{m_1}J_{(m_3,0)} +  \Delta
  W_{m_3}J_{(m_1,0)} - h\Delta W_{m_1} \Delta W_{m_3} \right).
\end{align*}
From \cref{thm:milstein} we can conclude that the method is (remarkably both for Strato\-no\-vich and It\^{o} SDEs) of strong order 1.5, since
\[ E \Phi(\tau_9) = E\phi(\tau_9) = \begin{cases} \frac{h^2}{2} & \text{if } m_2 = m_3 \\
  0 & \text{otherwise} \end{cases}, \qquad E \Phi(\tau_{12}) = E\phi(\tau_{12})= 0, \]
which follows from the following relations:
\[
E \Delta W_m^2 = h, \quad E \Delta W_m J_{(m,0)} = \frac{h^2}{2}, \qquad E J_{(m,0)}^2 = \frac{h^3}{3}.
\]

This is of course in accordance with the order result given (for the Stratonovich case) in \cite[Theorem 4.3]{milstein02nmf}.
\section{Quadratic invariants as simplifying assumptions, rootless trees}\label{sec:quadraticinvariants}
In the previous section, we demonstrated how the special structure of the SDEs can be exploited, in
the sense that the elementary differentials corresponding to certain trees are zero, and the
corresponding conditions for these trees do not need to be fulfilled. In this section, we will see
how for separable Stratonovich SDEs the algebraic relation between the method coefficients that implies the method to preserve
quadratic invariants creates certain equivalence classes of trees  in the sense that only one condition for each class has to be fulfilled. This
is a generalization of a result obtained for deterministic partitioned ODEs by Abia and Sanz-Serna
\cite{abia1993partitioned}. It is similar to the nonpartitioned case, which for ODEs is discussed by
Sanz-Serna and Abia \cite{sanzserna91ocf} and generalized to SDEs in Anmarkrud and Kv{\ae}rn{\o} \cite{anmarkrud17ocf}.
{The conditions for preserving quadratic invariants for non-separable equations with two
  partitionings have been developed by Hong, Xu and Wang
\cite{ma15ssp}\label{rev:24}}, {see also Ma and Ding \cite{ma15ssp}, both in the context of
symplectic methods.  Using similar ideas,
it is possible to find conditions for a more general partitioning. However, the extra freedom gained
by more partitionings is to some extent lost by the huge number of extra conditions that have to be
satisfied. In this paper, the discussion is restricted to separable equations with the number of partitions of \cref{eq:psde} to be two, i.e.}
\begin{subequations}
	\label{eq:psde2}
\begin{align}
	\dX^{(1)}(t) &=\sum_{\diffusionindex=0}^\diffusionnum g^{(1)}_\diffusionindex(X^{(2)}(t))\circ \dW_\diffusionindex(t), \\
	\dX^{(2)}(t) &=\sum_{\diffusionindex=0}^\diffusionnum g^{(2)}_\diffusionindex(X^{(1)}(t))\circ \dW_\diffusionindex(t),
\end{align}
\end{subequations}
{which is also the stochastic counterpart to the system discussed in \cite{abia1993partitioned}.}

We assume that the system has a quadratic invariant $I(X^{(1)},X^{(2)})={X^{(1)}}^\top DX^{(2)}$ for
a matrix $D$ of the appropriate dimension and arbitrary initial values $X^{(1)}(0),X^{(2)}(0)$.

{\begin{example}[Synchrotron oscillations]\label{ex:synchrotron}
Consider a stochastically perturbed Hamiltonian system of a pendulum (\cite{seesselberg94soo}, where we set $\lambda_A = 0$) with Hamiltonians
\begin{align}
	H_0(p,x,t)=\frac{p^2}{2},\qquad H_1(p,x,t)=\omega^2 \sin (x) \lambda_{Ph}.
	\label{eq:synchtrotron_hamiltonian}
\end{align}
The resulting SDE system
\begin{align}
	dp=-\omega^2 \sin (x)dt - \lambda_{Ph} \omega^2 \cos(x)\circ dW(t), \quad dx=pdt
	\label{eq:synchrotron_phase_noise}
\end{align}
preserves symplecticity, i.\,e.\ a quadratic invariant \cite[Section 4.1]{milstein04snf}.
\end{example}
}
As in \cite[chapter IV.2.2]{hairer2010geometric} one can prove the following theorem:
\begin{theorem}
	\label{thm:psrk_quad_invariants}
	If the coefficients of the {SPRK} method \eqref{eq:psrk} satisfy
	\begin{align}
\gamma_i^{(1,\diffusionindex_1)}\gamma_j^{(2,\diffusionindex_2)}&=\begin{multlined}[t]\gamma_j^{(2,\diffusionindex_2)}Z_{j,i}^{(1,\diffusionindex_1)}+\gamma_i^{(1,\diffusionindex_1)}Z_{i,j}^{(2,\diffusionindex_2)} \\\quad \forall i,j=1,\dots,s \text{ and } \forall \diffusionindex_1,\diffusionindex_2=0,1,\dots,\diffusionnum \label{eq:psrk_quad_invariant_b}
\end{multlined}
	\end{align}
	then it preserves quadratic invariants of \cref{eq:psde2} of the form $I(X^{(1)},X^{(2)})={X^{(1)}}^\top D X^{(2)}$.
\end{theorem}

When the conditions \eqref{eq:psrk_quad_invariant_b} are fulfilled, the number of order conditions
stated in \cref{thm:milstein} can be reduced.
The key ingredient in the analysis is the Butcher product of two trees
$u,v \in \tau_1,\tau_2,\dots,\tau_{\kappa} \in
T\setminus\{\emptyset_1,\dots,\emptyset_{\partitionnum}\}$. If
\begin{equation}\label{eq:uandv}
u = \left[ u_1, \dots,u_{\kappa_1}\right]_{\partitionindex_1,\diffusionindex_1}, \quad v = \left[ v_1, \dots,v_{\kappa_2}\right]_{\partitionindex_2,\diffusionindex_2},
\end{equation}
the \emph{Butcher product} is defined by
\begin{align*}
	u \circ v &= [u_1, \dots, u_{\kappa_1},v]_{\partitionindex_1,\diffusionindex_1}{,}
\end{align*}
see \cref{fig:butcherproduct}.
We can now state the following lemma:
\begin{lemma}
	\label{lem:nonrooted_exact}
	For Stratonovich SDEs of the form \eqref{eq:psde2} we have for all $u \in T_{\partitionindex_1,\diffusionindex_1}$ and $v \in T_{\partitionindex_2,\diffusionindex_2}$
	\begin{align*}
		\phi(u)(h)\cdot \phi(v)(h)&=\phi(u\circ v)(h) + \phi(v\circ u)(h)
	\end{align*}
	for all $\partitionindex_1,\partitionindex_2=1,2$ and $\diffusionindex_1,\diffusionindex_2=0,1,\dots,\diffusionnum$.
\end{lemma}
\begin{figure}[t]
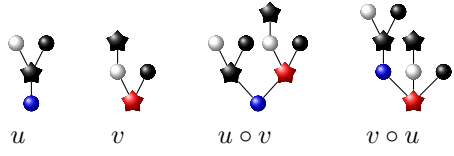

  \[
     \begin{array}{cccc}
        \quad \tu \quad &  \quad  \tv \quad  & \quad  \tuv \quad & \quad  \tvu \quad  \\ u & v & u\circ v & v\circ u
     \end{array}
  \]
  \caption{Demonstration of the Butcher product (with real colors and shapes instead of indices).\label{fig:butcherproduct}}
\end{figure}
\begin{proof}
  As for $u$ and $v$ given by \cref{eq:uandv} it holds
  \[ \phi(u)(t) = \int_0^t\prod_{k_1=1}^{\kappa_1}\phi(u_{k_1})(s)\circ \dW_{\diffusionindex_1}(s), \qquad
     \phi(v)(t) = \int_0^t\prod_{k_2=1}^{\kappa_2}\phi(v_{k_2})(s)\circ \dW_{\diffusionindex_2}(s),
  \]
  the product rule for Stratonovich integrals gives
  \begin{align*}
	\phi(u)(t)\cdot \phi(v)(t)&=\int_0^t
	\phi(v)(s)\prod_{\treeindex_1=1}^{\kappa_1}\phi(u_{\treeindex_1})(s)\circ \dW_{\diffusionindex_1}(s) \\
	&+\int_0^t \phi(u)(s)\prod_{\treeindex_2=1}^{\kappa_2}\phi(v_{\treeindex_2})(s)\circ \dW_{\diffusionindex_2}(s) \\
	&= \phi(u \circ v)(t)+\phi(v \circ u)(t).
\end{align*}
\end{proof}
For $\partitionindex_1\neq\partitionindex_2$ this rule also holds for the weight functions of the numerical solution as stated in the following lemma:
\begin{lemma}
	\label{lem:nonrooted_num}
	If an SPRK method of the form \eqref{eq:psrk} satisfies the conditions in \cref{thm:psrk_quad_invariants},
	then for all $u \in T_{1,\diffusionindex_1}$ and $v \in T_{2,\diffusionindex_2}$
	\begin{align*}
	\Phi(u)(h)\cdot \Phi(v)(h)=\Phi(u \circ v)(h) + \Phi(v \circ u)(h)
	\end{align*}
	for all $\diffusionindex_1,\diffusionindex_2=0,1,\dots,m$.
\end{lemma}

\begin{proof}
Multiply both sides of \cref{eq:psrk_quad_invariant_b} in \cref{thm:psrk_quad_invariants} by
\[\prod_{\treeindex_1=1}^{\kappa_1}\Psi_i(u_{\treeindex_1})(h)\prod_{\treeindex_2=1}^{\kappa_2}\Psi_j(v_{\treeindex_2})(h)\] and sum over all $i,j=1,\dots,s$,
\begin{multline*}
	\left(\sum_{i}\gamma_i^{(1,\diffusionindex_1)}\prod_{\treeindex_1=1}^{\kappa_1}\Psi_i(u_{\treeindex_1})(h)\right)\left(\sum_j\gamma_j^{(2,\diffusionindex_2)}\prod_{\treeindex_2=1}^{\kappa_2}\Psi_j(v_{\treeindex_2})(h)\right)= \\ \sum_{i,j}\left(\gamma_j^{(2,\diffusionindex_2)}Z_{j,i}^{(1,\diffusionindex_1)}+\gamma_i^{(1,\diffusionindex_1)}Z_{i,j}^{(2,\diffusionindex_2)}\right)\prod_{\treeindex_1=1}^{\kappa_1}\Psi_i(u_{\treeindex_1})(h)\prod_{\treeindex_2=1}^{\kappa_2}\Psi_j(v_{\treeindex_2})(h).
\end{multline*}
Using \cref{eq:num_b_h,eq:num_b_y} completes the proof.
\end{proof}
Let $TS$ be the set of trees for which a node of one shape will never be directly followed by a node of the same shape. For separable SDEs, all elementary differentials of trees in $T\setminus TS$ vanish.
Given a $\tau \in TS$, let $\hat{\tau}$ be
the corresponding unrooted tree, and let
$\hatTS(\hat{\tau}) \subset TS$ be the set of trees obtained from $\hat{\tau}$ by assigning one of the
nodes as the root, see \cref{fig:unrooted} for an illustration.
\begin{figure}[t]
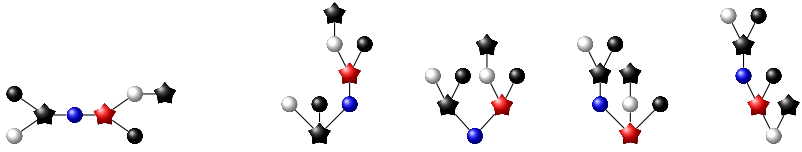

  \[\trootless  \qquad \qquad \tuvl \qquad  \tuv \qquad \tvu \qquad \tuvr   \]
  \caption{An unrooted tree $\hat{\tau}$ to the left, and some of the trees in $\hatTS(\hat{\tau})$.
  The figure also illustrates the process described in the proof of \cref{thm:ordcond_quad}.}
  \label{fig:unrooted}
\end{figure}

\begin{theorem} \label{thm:ordcond_quad}
	Assume that \cref{eq:psrk_quad_invariant_b} is satisfied.
	Let $\hat{\tau} \in \hatTS$ be an unrooted tree
	of order $q\leq p$. If ${\phi}(\tau)(h)=\Phi(\tau)(h) +
	\mathcal{O}(h^{p+\frac{1}{2}})$ for one rooted tree $\tau \in \hatTS(\hat{\tau})$ and all
	rooted trees of order less than $q$, then it holds that
	${\phi}(\tau)(h)=\Phi(\tau)(h) + \mathcal{O}(h^{p+\frac{1}{2}})$ for all $\tau \in \hatTS(\hat{\tau})$. {Similarly, if $E\phi(\tau)(h)=E\Phi(\tau)(h) +
	\mathcal{O}(h^{p+\frac{1}{2}})$ for one rooted tree $\tau \in \hatTS(\hat{\tau})$ and $\phi(\tau)(h)=\Phi(\tau)(h) +
	\mathcal{O}(h^{p+\frac{1}{2}})$ for all
	rooted trees of order less than $q$, then it holds that
	$E \phi(\tau)(h)=E \Phi(\tau)(h) + \mathcal{O}(h^{p+\frac{1}{2}})$ for all $\tau \in \hatTS(\hat{\tau})$.}
\end{theorem}
\begin{proof}
	The argument is the same as for the non-partitioned case proven in \cite{anmarkrud17ocf}.
	For trees with one node the theorem is trivially true.
	Let $\hat{\tau}$ be an unrooted tree of order $q$ and two or more nodes, and let $\tau$ be a corresponding rooted tree $\tau \in \hatTS(\hat{\tau})$.
	Pick one branch $v$ from the root of $\tau$ and let the remaining part of $\tau$ be $u$, so that $\tau = u \circ v$. Then $u$ and $v$ have roots of different shapes, and \cref{lem:nonrooted_num} applies.
	Clearly, the orders of $u$ and $v$ are less than the order of $\tau$, and by \cref{lem:nonrooted_exact} and the assumptions of the theorem we then have
\begin{align*}
  {\phi}(v\circ u)(h) & = \Phi(v\circ u)(h) + \mathcal{O}(h^{p+\frac{1}{2}})
  \intertext{{respectively}}
  {E\phi(v\circ u)(h) }& {= E\Phi(v\circ u)(h) + \mathcal{O}(h^{p+\frac{1}{2}})}
  {.}
\end{align*}

Because the choice of branch $v$ was arbitrary, this means that this condition is satisfied for all trees with the same graph as $\tau$, but with a root shifted to one of its neighboring nodes. A repeated use of this argument proves the
  result. The process is illustrated in \cref{fig:unrooted}.
\end{proof}
{
\begin{example}\label{ex:HongsStoermerVerletMethod}
A simple example for an SPRK method fulfilling \eqref{eq:psrk_quad_invariant_b} is the stochastic St\"{o}rmer-Verlet method \cite{hong15poq}, given by the following tableau:
\begin{equation*}
  \renewcommand{\arraystretch}{1.5}
  \begin{array}{cc|cc}
    0 & 0 & 0 &  0 \\
    \frac{h}{2} & \frac{h}{2} &  \frac{\Delta W_\diffusionindex}{2}  & \frac{\Delta W_\diffusionindex}{2}  \\
    \hline
    \frac{h}{2} & 0 &  \frac{\Delta W_\diffusionindex}{2}  & 0 \\
    \frac{h}{2} & 0 &  \frac{\Delta W_\diffusionindex}{2} & 0 \\
    \hline \hline
    \frac{h}{2} & \frac{h}{2} & \frac{\Delta W_\diffusionindex}{2}  &  \frac{\Delta W_\diffusionindex}{2} \\
    \hline
    \frac{h}{2} & \frac{h}{2} & \frac{\Delta W_\diffusionindex}{2}  &  \frac{\Delta W_\diffusionindex}{2} \\
  \end{array}.
\end{equation*}
To analyze its order, note that due to \eqref{eq:psde2} being separable, the elementary differentials corresponding to trees for which a node of shape $q$ ($q\in\{1,2\}$) is followed by a node of the same shape vanish. Taking into account \cref{thm:ordcond_quad}, the remaining trees $\tau$ with $\rho(\tau)\leq\frac32$ to be considered are
then (for $q_1,q_2\in\{1,2\}$ with $q_1\neq q_2$)
\begin{equation} \label{eq:trees_separable_quadratic}
  \begin{array}{@{}c|@{}c@{}c@{}c@{}c@{}c@{}} \\
    \text{i}&  1 & 2 & 3 & 5 &7
     \\ \hline
  \tau_i & \maketree{\rn{white}{$_{q_1,m_1}$}}&
  \maketree{\rn{black}{$_{q_1}$}}&
  \maketree{\rn{white}{$_{1,m_1}$} child{\n{white}{$_{2,m_2}$}}} &
  \maketree{\rn{white}{$_{{q}_1,m_1}$} child{\n{white}{$_{{q}_2,m_2}$} child{\n{white}{$_{{q}_1,m_3}$}}}}&
  \maketree{\rn{black}{$_{q_1}$} child{\n{white}{$_{q_2,m_2}$}}}
\end{array}
\end{equation}
It is straightforward to show that $\phi(\tau) = \Phi(\tau)$ for trees $\tau_1,\tau_2$ in \cref{eq:trees_separable_quadratic}, and that this is {true for $\tau_3$ only} if $m_1=m_2$, while it holds also for $m_1\neq m_2$ that $E\phi(\tau_3) = E\Phi(\tau_3)$. Thus, for SDEs \eqref{eq:psde2} with multidimensional noise, the above method is only of order $\frac12$. For scalar noise, it holds also $\phi(\tau_7) = \Phi(\tau_7)$ and $E\phi(\tau_5) = E\Phi(\tau_5)$, but not $\phi(\tau_5) = \Phi(\tau_5)$, so in this case the method is of order 1. Similarly, it follows that the method is of weak order one for multidimensional noise.
\end{example}
}
\begin{table}
    \centering
    {
    \caption{{Total number of trees up to order 2 for SDEs with scalar noise and two partitionings. Here, column (all) gives the number when assuming no restrictions, column (s.p.) the number counting only trees with nonzero elementary differentials for separable problems, and column (q.i.) the number of independent conditions if additionally the quadratic invariant condition is satisfied.}}
    \label{tab:tree_count}
    \begin{tabular}{crrr}
    	\hline\noalign{\smallskip}
        $\rho$ & all & s.p. & q.i. \\
    	\noalign{\smallskip}\hline\noalign{\smallskip}
        0.5 & 2 & 2 & 2 \\
        1 & 6 & 4 & 3 \\
        1.5 & 22 & 8 & 4 \\
        2 & 92 & 20 & 9 \\
       	\noalign{\smallskip}\hline\noalign{\smallskip}
     Sum & 122 & 34 & 18 \\
    	\noalign{\smallskip}\hline
    \end{tabular}
    }
\end{table}
{The content of this section demonstrates the potential of the B-series formulation, in the
  sense that theory developed and well understood for ODEs here directly can be applied for SDEs. It also demonstrates how to use the particular structure of the problem at hand to reduce the set of order conditions that have to be satisfied. This is demonstrated in \cref{tab:tree_count}. But even if the number of conditions is reduced in this case, the conditions in \cref{thm:psrk_quad_invariants} are still so restrictive that it is far from a trivial task to construct higher order methods.}
\section{Conclusion}
In this paper we have developed a general B-series theory for {SPRK}
methods. Such methods are rarely applied to general SDEs, they are more likely to be constructed
for SDEs with certain structure. We have therefore emphasized how the general theory, summarized in
\cref{thm:milstein}, can be modified to cover a few common examples of such cases. We hope other researchers can find the theory useful for constructing new methods, or to prove order results for applying existing methods to broader classes than originally constructed for.

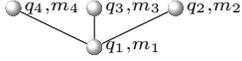
\newcounter{Treenumber}\stepcounter{Treenumber}
\newcommand{\tabentry}[5]{\arabic{Treenumber}\stepcounter{Treenumber}\gentabentry{#1}{#2}{#3}{#4}{#5}}
\newcommand{\headentry}[5]{{No.}\gentabentry{#1}{#2}{#3}{#4}{#5}}
\newcommand{\gentabentry}[5]{&#1&#2&{$\begin{array}{c}\noalign{\vskip 1mm}#3\\[1mm]\hdashline \noalign{\vskip 1mm}   #4\end{array}$}&$\  #5 $}
\newcommand{\gind}[1]{g_{\diffusionindex_#1}^{(\partitionindex_#1)}}
\begin{landscape}
\appendix
\section{Table of trees and corresponding functions}\label{appendix}
\begin{longtable}{@{}r@{}ccccc}
\headentry{$\tau$}{$\rho(\tau)$}{\phi(\tau)}{\Phi(\tau)}{F(\tau)}
\\
\hline
\tabentry{\maketree{\sr}}{0.5}{W_{\diffusionindex_1}(h)}{\sum_{i}\limits \gamma_i^{(q_1,m_1)}}{\gind{1}}
\\
\tabentry{\maketree{\dr}}{1}{h}{\sum_{i}\limits \gamma_i^{(q_1,0)}}{g_{0}^{(\partitionindex_1)}}
\\
\tabentry{\maketree{\sr
	child{\sn{2}}
}
 }{1}{\int_0^h W_{\diffusionindex_2}(s) \star\dW_{\diffusionindex_1}(s)}{\sum_{i,j}\limits \gamma_i^{(q_1,m_1)} Z_{i,j}^{(q_2,m_2)}}{ (D_{\partitionindex_2}\gind{1})\gind{2}}
\\
\tabentry{\maketree{\sr
	child{\sn{3}}
	child{\sn{2}}
}}{1.5}{\int_0^h W_{\diffusionindex_2}(s)W_{\diffusionindex_3}(s)\star\dW_{\diffusionindex_1}(s)}{\sum_{i,j,k}\limits \gamma_i^{(q_1,m_1)}Z_{i,j}^{(q_2,m_2)}Z_{i,k}^{(q_3,m_3)}}{(D_{\partitionindex_2\partitionindex_3}\gind{1})(\gind{2},\gind{3})}
\\
\tabentry{\maketree{\sr
	child{\sn{2}
	child{\sn{3}}}
}}{1.5}{\int_0^h \int_0^s W_{\diffusionindex_3}(s_1)\star\dW_{\diffusionindex_2}(s_1)\star\dW_{\diffusionindex_1}(s)}{\sum_{i,j,k}\limits \gamma_i^{(q_1,m_1)}Z_{i,j}^{(q_2,m_2)}Z_{j,k}^{(q_3,m_3)}}{(D_{\partitionindex_2}\gind{1})(D_{\partitionindex_3}\gind{2})\gind{3}}
\\
\tabentry{\maketree{\sr
	child{\dn{2}}
}}{1.5}{\int_0^h s \star\dW_{\diffusionindex_1}(s)}{\sum_{i,j}\limits \gamma_i^{(q_1,m_1)} Z_{i,j}^{(q_2,0)}}{(D_{\partitionindex_2}\gind{1})g_0^{(\partitionindex_2)}}
\\
\tabentry{\maketree{\dr
	child{\sn{2}}
}}{ 1.5 }{\int_0^h W_{\diffusionindex_2}(s)\ds}{\sum_{i,j}\limits \gamma_i^{(q_1,0)} Z_{i,j}^{(q_2,m_2)}}{(D_{\partitionindex_2}g_0^{(\partitionindex_1)})\gind{2}}
\\
\tabentry{\maketree{\dr
	child{\dn{2}}
}}2{\frac{1}{2}h^2}{\sum_{i,j}\limits \gamma_i^{(q_1,0)} Z_{i,j}^{(q_2,0)}}{(D_{\partitionindex_2}g_0^{(\partitionindex_1)})g_0^{(\partitionindex_2)}}
\\
\tabentry{\maketree{\dr
	child{\sn{3}}
	child{\sn{2}}
}}2{\int_0^h W_{\diffusionindex_2}(s)W_{\diffusionindex_3}(s)\ds}{\sum_{i,j,k}\limits \gamma_i^{(q_1,0)}Z_{i,j}^{(q_2,m_2)}Z_{i,k}^{(q_3,m_3)}}{(D_{\partitionindex_2\partitionindex_3}g_0^{(\partitionindex_1)})(\gind{2},\gind{3})}
\\
\tabentry{\maketree{\sr
	child{\dn{3}}
	child{\sn{2}}
}}
2{\int_0^h W_{\diffusionindex_2}(s)s \star\dW_{\diffusionindex_1}(s)}{\sum_{i,j,k}\limits \gamma_i^{(q_1,m_1)}Z_{i,j}^{(q_2,m_2)}Z_{i,k}^{(q_3,0)}}{(D_{\partitionindex_2\partitionindex_3}\gind{1})(\gind{2},g_0^{(\partitionindex_3)})}
\\
\tabentry{\maketree{\dr
	child{\sn{2}
	child{\sn{3}}}
}}
2{\int_0^h\int_0^s W_{\diffusionindex_3}(s_1)\star\dW_{\diffusionindex_2}(s_1) \star\dW_{\diffusionindex_1}(s)}{\sum_{i,j,k}\limits \gamma_i^{(q_1,0)}Z_{i,j}^{(q_2,m_2)}Z_{j,k}^{(q_3,m_3)}}{(D_{\partitionindex_2}g_0^{(\partitionindex_1)})(D_{\partitionindex_2}\gind{2})\gind{3}}
\\
\tabentry{\maketree{\sr
	child{\dn{2}
	child{\sn{3}}}
}}2{\int_0^h\int_0^s W_{\diffusionindex_3}(s_1)\ds_1 \star\dW_{\diffusionindex_1}(s)}{\sum_{i,j,k}\limits \gamma_i^{(q_1,m_1)}Z_{i,j}^{(q_2,0)}Z_{j,k}^{(q_3,m_3)}}{(D_{\partitionindex_2}\gind{1})(D_{\partitionindex_3}g_0^{(\partitionindex_2)})\gind{3}}
\\
\tabentry{\maketree{\sr
	child{\sn{2}
	child{\dn{3}}}
}} 2{\int_0^h\int_0^s s_1 \star\dW_{\diffusionindex_2}(s_1) \star\dW_{\diffusionindex_1}(s)}{\sum_{i,j,k}\limits \gamma_i^{(q_1,m_1)}Z_{i,j}^{(q_2,m_2)}Z_{j,k}^{(q_3,0)}}{(D_{\partitionindex_2}\gind{1})(D_{\partitionindex_3}\gind{2})g_0^{(\partitionindex_3)}}
\\
\tabentry{\maketree{\sr
	child{ \sn{2}
		child{ \sn{3}
		    child{ \sn{4}
			}
		}
	}
}} 2 {\int_0^h \int_0^s \int_0^{s_1}\dW_{\diffusionindex_4}(s_2)\star\dW_{\diffusionindex_3}(s_2) \star\dW_{\diffusionindex_2}(s_1) \star\dW_{\diffusionindex_1}(s)}{\sum_{i,j,k,l}\limits \gamma_i^{(q_1,m_1)}Z_{i,j}^{(q_2,m_2)}Z_{j,k}^{(q_3,m_3)}Z_{k,l}^{(q_4,m_4)}}{(D_{\partitionindex_2}\gind{1})(D_{\partitionindex_3}\gind{2})(D_{\partitionindex_4}\gind{3})\gind{4}}
\\
\tabentry{\maketree{\sr
	child{ \sn{3}
	    child{ \sn{4}
		}
	}
	child{ \sn{2}}
}}2 {\int_0^h W_{\diffusionindex_2}(s) \int_0^{s} W_{\diffusionindex_4}(s_1) \circ\dW_{\diffusionindex_3}(s_1)\circ\dW_{\diffusionindex_1}(s)}{\sum_{i,j,k,l}\limits \gamma_i^{(q_1,m_1)}Z_{i,j}^{(q_2,m_2)}Z_{i,k}^{(q_3,m_3)}Z_{k,l}^{(q_4,m_4)}}{(D_{\partitionindex_2\partitionindex_3}\gind{1})(\gind{2},(D_{\partitionindex_4}\gind{3})\gind{4})}
\\
\tabentry{\maketree{\sr
	child{\sn{2}}
	child{\sn{3}}
	child{\sn{4}}
}} 2{\int_0^h W_{\diffusionindex_4}(s)W_{\diffusionindex_3}(s)W_{\diffusionindex_2}(s) \circ\dW_{\diffusionindex_1}(s)}{\sum_{i,j,k,l}\limits \gamma_i^{(q_1,m_1)}Z_{i,j}^{(q_2,m_2)}Z_{i,k}^{(q_3,m_3)}Z_{i,l}^{(q_4,m_4)}}{(D_{\partitionindex_2\partitionindex_3\partitionindex_4}\gind{1})(\gind{2},\gind{3},\gind{4})}
\\
\tabentry{\maketree{\sr
	child{\sn{2}
	child{\sn{3}}
	child{\sn{4}}}
}}
2{\int_0^h \int_0^s W_{\diffusionindex_4}(s_1)W_{\diffusionindex_3}(s_1)\circ\dW_{\diffusionindex_2}(s_1) \circ\dW_{\diffusionindex_1}(s)}{\sum_{i,j,k,l}\limits \gamma_i^{(q_1,m_1)}Z_{i,j}^{(q_2,m_2)}Z_{j,k}^{(q_3,m_3)}Z_{j,l}^{(q_4,m_4)}}
{(D_{\partitionindex_2}\gind{1})(D_{\partitionindex_3\partitionindex_4}\gind{2})(\gind{3},\gind{4})}
\end{longtable}
\end{landscape}

%\bibliographystyle{amsplain}
%\bibliography{stok}
\providecommand{\bysame}{\leavevmode\hbox to3em{\hrulefill}\thinspace}
\providecommand{\MR}{\relax\ifhmode\unskip\space\fi MR }
\providecommand{\MRhref}[2]{%
  \href{http://www.ams.org/mathscinet-getitem?mr=#1}{#2}
}
\providecommand{\href}[2]{#2}

\end{document}